\documentclass[12pt]{amsart}

\usepackage{amssymb}

\newtheorem{theorem}{Theorem}[section]

\newtheorem{lemma}[theorem]{Lemma}
\newtheorem{remark}[theorem]{Remark}

\numberwithin{equation}{section}

\begin{document}
\baselineskip=15.5pt

\title[Irreducible representations of cocompact lattices of $\text{SL}(2,{\mathbb
C})$]{On the vector bundles associated to the irreducible representations
of cocompact lattices of $\text{SL}(2,{\mathbb C})$}

\author[I. Biswas]{Indranil Biswas}

\address{School of Mathematics, Tata Institute of Fundamental
Research, Homi Bhabha Road, Bombay 400005, India}

\email{indranil@math.tifr.res.in}

\author[A. Mukherjee]{Avijit Mukherjee}

\address{Department of Physics,
Jadavpur University, Raja S. C. Mullick Road, Jadavpur,
Kolkata 700032, India}

\email{avijit00@gmail.com}



\date{}

\begin{abstract}
In this continuation of \cite{BM}, we prove the following: Let
$\Gamma\, \subset\, \text{SL}(2,{\mathbb C})$ be a cocompact lattice, and
let $\rho\,:\, \Gamma\, \longrightarrow\, \text{GL}(r,{\mathbb C})$ be
an irreducible representation. Then the holomorphic vector bundle
$E_\rho\, \longrightarrow\, \text{SL}(2,{\mathbb C})/\Gamma$ associated
to $\rho$ is polystable. The compact complex manifold
$\text{SL}(2,{\mathbb C})/\Gamma$ has natural Hermitian structures; the
polystability of $E_\rho$ is with respect to these natural Hermitian
structures. In \cite{BM} it was shown that if $\rho(\Gamma)
\,\subset\, \text{U}(r)$, then $E_\rho$ equipped with the Hermitian
structure given by $\rho$, and $\text{SL}(2,{\mathbb C})/\Gamma$
equipped with a natural Hermitian structure, together produce a solution
of the Strominger system of equations. A polystable vector bundle also has
a natural Hermitian structure, which is known as the Hermitian--Yang--Mills
structure. It would be interesting to find similar applications of the
Hermitian--Yang--Mills structure on the above polystable vector bundle $E_\rho$.
We show that the polystable vector bundle $E_\rho$ is not stable in general.
\end{abstract}

\maketitle

\section{Introduction}\label{sec1}

We first recall the set--up, and some results, of \cite{BM}. Let
$$
\Gamma\, \subset\, \text{SL}(2,{\mathbb C})
$$
be a discrete cocompact subgroup. Fixing a $\text{SU}(2)$-invariant
Hermitian form on the Lie algebra
$sl(2,{\mathbb C})$, we get a Hermitian structure
$h$ on the compact complex manifold $M\,:=\, \text{SL}(2,{\mathbb C})/\Gamma$.
The $(1\, ,1)$--form $\omega_h$ on $M$ associated to $h$ satisfies the
identity $d\omega^2_h\,=\, 0$. Take any homomorphism
$$
\rho\,:\, \Gamma\, \longrightarrow\, \text{GL}(r,{\mathbb C})\, .
$$
This $\rho$ produces a holomorphic vector bundle $E_\rho$ of rank $r$ on $M$
equipped with a flat holomorphic connection $\nabla^\rho$.
The homomorphism $\rho$ is called irreducible if $\rho(\Gamma)$ is not contained
in some proper parabolic subgroup of $\text{GL}(r,{\mathbb C})$.

If $\rho(\Gamma)\, \subset\, \text{U}(r)$, then $E_\rho$ is equipped with
a Hermitian structure $H^\rho$ such that the associated Chern connection is
$\nabla^\rho$.

If
\begin{itemize}
\item $\rho(\Gamma)\, \subset\, \text{U}(r)$ and

\item $\rho(\Gamma)$ is irreducible,
\end{itemize}
then the quadruple $(M\, ,h\, ,E_\rho\, ,H^\rho)$ satisfies the Strominger
system of equations \cite[Theorem 4.6]{BM}. In particular, the
vector bundle $E_\rho$ is stable \cite[Proposition 4.5]{BM}.

Now assume that $\rho$ is irreducible, but do \textit{not} assume that
$\rho(\Gamma)\, \subset\, \text{U}(r)$. Our aim here is to prove the
following (see Theorem \ref{thm1}):

\textit{The holomorphic vector bundle $E_\rho$ is polystable with respect
to the Hermitian structure $h$ on $M$.}

It is known that under some minor condition, the group $\Gamma$ admits some
free groups of more than one generators as quotients \cite[p. 3393,
Theorem 2.1]{La}. Therefore, there are many examples of pairs $(\Gamma\, ,\rho)$
of the above type satisfying the irreducibility condition.

Since $E_\rho$ is polystable, the holomorphic vector bundle $E_\rho$ has
an Hermitian--Yang--Mills structure ${\mathcal H}^\rho$ \cite{LY} (see also
\cite{Bu}). It may be worthwhile to investigate this Hermitian
structure ${\mathcal H}^\rho$. We should clarify that ${\mathcal H}^\rho$
need not be flat. An Hermitian--Yang--Mills structure on a polystable vector
bundle with vanishing Chern classes over a compact K\"ahler manifold is flat,
but $M$ is not K\"ahler.

It is natural to ask whether the polystable vector bundle $E_\rho$ is stable.
If we take $\rho$ to be the inclusion of $\Gamma$ in $\text{SL}(2,{\mathbb C})$,
then $\rho$ is irreducible, but the associated holomorphic vector bundle $E_\rho$
is holomorphically trivial, in particular, $E_\rho$ is not stable
(see Lemma \ref{lem1} for the details).

Infinitesimal deformations of the complex structure of $M$ are investigated
in \cite{Ra}.

\section{Polystability of associated vector bundle}

The Lie algebra of $\text{SL}(2,{\mathbb C})$, which will be denoted by
$sl(2,{\mathbb C})$, is the space of complex $2\times 2$ matrices of trace zero.
Consider the adjoint action of $\text{SU}(2)$ on $sl(2,{\mathbb C})$.
Fix an inner product $h_0$ on $sl(2,{\mathbb C})$ preserved by this
action; for example, we may take the Hermitian form
$(A\, ,B)\,\longmapsto\, \text{trace}(AB^*)$ on $sl(2,{\mathbb C})$.
Let $h_1$ be the Hermitian structure on $\text{SL}(2,{\mathbb C})$
obtained by right--translating the Hermitian form $h_0$ on
$T_{\rm Id}\text{SL}(2,{\mathbb C})\,=\, sl(2,{\mathbb C})$.

Let $\Gamma$ be a cocompact lattice in $\text{SL}(2,{\mathbb C})$. So
$\Gamma$ is a discrete subgroup of $\text{SL}(2,{\mathbb C})$ such that the
quotient
\begin{equation}\label{e1}
M\,:=\, \text{SL}(2,{\mathbb C})/\Gamma
\end{equation}
is compact. This $M$ is a compact complex manifold of complex
dimension three. The left--translation action of $\text{SL}(2,
{\mathbb C})$ on itself descends to an action of $\text{SL}(2,{\mathbb C})$ on
$M$. We will call this action of $\text{SL}(2,{\mathbb C})$
on $M$ the \textit{left--translation action}.
The Hermitian structure $h_1$ on $\text{SL}(2,{\mathbb C})$ descends to an
Hermitian structure on $M$. This descended Hermitian structure on $M$ will be
denoted by $h$. Let $\omega_h$ be the $C^\infty$ $(1\, ,1)$--form on $M$
associated to $h$. Then
$$
d\omega^2_h\,=\, 0
$$
\cite[Corollary 4.1]{BM}.

For a torsionfree nonzero coherent analytic sheaf $F$ on $M$, define
$$
\text{degree}(F)\,:=\, \int_M c_1(F)\wedge \omega^2_h\, \in\, {\mathbb R}
~\ \text{ and }~\ \mu(F)\,:=\, \frac{\text{degree}(F)}{\text{rank}(F)}
\, \in\, {\mathbb R}\, .
$$
A torsionfree nonzero coherent analytic sheaf $F$ on $M$ is called \textit{stable}
(respectively, \textit{semistable}) if for every coherent analytic subsheaf
$$
V\,\subset\, F
$$
such the $\text{rank}(V)\, \in\, [1\, ,\text{rank}(F)-1]$ and the
quotient $F/V$ is torsionfree, the inequality
$$
\mu(V)\,< \, \mu(F)~\,\ \text{(respectively,~}\,
\mu(V)\,\leq \, \mu(F){\rm )}
$$
holds (see \cite[Ch.~V, \S~7]{Ko}). A torsionfree nonzero coherent analytic
sheaf $F$ on $M$ is called \textit{polystable} if it is semistable and is
isomorphic to a direct sum of stable sheaves.

\begin{remark}\label{rem2}
{\rm Since a polystable coherent analytic sheaf $F$ is semistable, if $F\,=\,
\bigoplus_{i=1}^\ell F_i$, then $\mu(F_i)\,=\,\mu(F)$ for all $i$.}
\end{remark}

Take any homomorphism
\begin{equation}\label{e2}
\rho\,:\, \Gamma\, \longrightarrow\, \text{GL}(r,{\mathbb C})\, .
\end{equation}
Let $(E_\rho\, ,\nabla^\rho)$ be the flat holomorphic vector bundle
of rank $r$ over $M$ associated
to the homomorphism $\rho$. We recall that the total space of
$E_\rho$ is the quotient of $\text{SL}(2,{\mathbb C})
\times {\mathbb C}^r$ where two points $$(z_1\, ,v_1)\, , (z_2\, ,v_2)\,\in\,
\text{SL}(2,{\mathbb C})\times {\mathbb C}^r$$ are identified if there is an
element $\gamma\, \in\, \Gamma$ such that $z_2\,=\, z_1\gamma$ and $v_2\,=\,
\rho(\gamma^{-1})(v_1)$. The trivial connection on the trivial vector bundle
$\text{SL}(2,{\mathbb C})\times {\mathbb C}^r\,\longrightarrow\,
\text{SL}(2,{\mathbb C})$ of rank $r$
descends to the connection $\nabla^\rho$. The left--translation
action of $\text{SL}(2,{\mathbb C})$ on $\text{SL}(2,{\mathbb C})$ and the trivial
action of $\text{SL}(2,{\mathbb C})$ on ${\mathbb C}^r$ together define an action
of $\text{SL}(2,{\mathbb C})$ on $\text{SL}(2,{\mathbb C})
\times {\mathbb C}^r$. This action
of $\text{SL}(2,{\mathbb C})$ on $\text{SL}(2,{\mathbb C})
\times {\mathbb C}^r$ descends to an action
\begin{equation}\label{ta}
\tau\, :\, \text{SL}(2,{\mathbb C})\times E_\rho\,\longrightarrow\,E_\rho
\end{equation}
of $\text{SL}(2,{\mathbb C})$
on the vector bundle $E_\rho$. The action $\tau$ in \eqref{ta}
is clearly a lift of the left--translation action of $\text{SL}(2,{\mathbb C})$ on $M$.

The homomorphism $\rho$ in \eqref{e2} is called \textit{reducible} if there a
nonzero linear subspace $S\,\subsetneq\, {\mathbb C}^r$ such that
$\rho(\Gamma)(S)\, =\, S$.
The homomorphism $\rho$ is called \textit{irreducible} if it is not reducible.

\begin{theorem}\label{thm1}
Assume that the homomorphism $\rho$ in \eqref{e2} is irreducible.
Then the corresponding holomorphic vector bundle $E_\rho$ is polystable.
\end{theorem}

\begin{proof}
Since $E_\rho$ has a flat connection, the Chern class
$c_1(\det E_\rho)\,=\, c_1(E_\rho)\,\in\, H^2(M,\, {\mathbb R})$ vanishes.
Hence we have $\text{degree}(E_\rho)\,=\, 0$ (see
\cite[Lemma 4.2]{BM}).

We will first show that $E_\rho$ is semistable. Assume that $E_\rho$ is not
semistable. Let
\begin{equation}\label{g1}
0\, \subset\, W_1\,\subset\, \cdots \,\subset\, W_{\ell-1}\,\subset\,
W_\ell\,=\, E_\rho
\end{equation}
be the Harder--Narasimhan filtration $E_\rho$; see \cite{Br} for the
construction of the Harder--Narasimhan filtration of vector bundles on
compact complex manifolds. Since $E_\rho$ is not semistable, we have
$\ell\, \geq\, 2$ and $W_1\,\not=\, 0$.

Consider the action $\tau$ of $\text{SL}(2,{\mathbb C})$ on $E_\rho$ constructed
in \eqref{ta}. From the uniqueness of the Harder--Narasimhan filtration it
follows immediately that $\tau(\{g\}\times W_1)\,=\, W_1$
for every $g\, \in\, \text{SL}(2,{\mathbb C})$. Therefore, we have 
\begin{equation}\label{g2}
\tau(\text{SL}(2,{\mathbb C})\times W_1)\,=\, W_1\, .
\end{equation}
Let $C(W_1)\, \subsetneq\, M$ be the closed subset over which $W_1$ fails to
be locally free. Since $\tau$ is a lift of the left--translation action of
$\text{SL}(2,{\mathbb C})$ on $M$, from \eqref{g2} we conclude that
$C(W_1)$ is preserved by the left--translation action of
$\text{SL}(2,{\mathbb C})$ on $M$. As the left--translation action of
$\text{SL}(2,{\mathbb C})$ on $M$ is transitive, it follows that
$C(W_1)$ is the empty set. Therefore, $W_1$ is a holomorphic vector bundle
on $M$. Similarly, the closed proper subset of $M$ over which $W_1$ fails to be
a subbundle of $E_\rho$ is preserved the left--translation action of
$\text{SL}(2,{\mathbb C})$ on $M$. Hence this subset is empty, and
$W_1$ is a holomorphic subbundle of $E_\rho$.

We will show that the flat connection $\nabla^\rho$ on $E_\rho$ preserves
the subbundle $W_1$ in \eqref{g1}.

To show that $\nabla^\rho$ preserves $W_1$, first note that
the flat sections of the trivial connection on the trivial vector bundle
$\text{SL}(2,{\mathbb C})\times {\mathbb C}^r\,\longrightarrow\,
\text{SL}(2,{\mathbb C})$ are of the form
$$
\text{SL}(2,{\mathbb C})\,\longrightarrow\,\text{SL}(2,{\mathbb C})\times
{\mathbb C}^r\, , ~ \  g\, \longmapsto\, (g\, ,v_0)\, ,
$$
where $v_0\,\in\, {\mathbb C}^r$ is independent of $g$. On the other hand,
the image of such a section is an orbit for the action of
$\text{SL}(2,{\mathbb C})$ on $\text{SL}(2,{\mathbb C})\times
{\mathbb C}^r$; recall that the action of $\text{SL}(2,{\mathbb C})$
on $\text{SL}(2,{\mathbb C})\times {\mathbb C}^r$ is the diagonal one
for the left--translation action of $\text{SL}(2,{\mathbb C})$
on itself and the trivial action of $\text{SL}(2,{\mathbb C})$
on ${\mathbb C}^r$ (see the construction of $\tau$ in \eqref{ta}). Also,
recall that the connection $\nabla^\rho$ on $E_\rho$ is the descent
of the trivial connection on the trivial vector bundle
$\text{SL}(2,{\mathbb C})\times {\mathbb C}^r\,\longrightarrow\,
\text{SL}(2,{\mathbb C})$. Combining these, from \eqref{g2} we conclude
that $\nabla^\rho$ preserves $W_1$.

The homomorphism $\rho$ is given to be irreducible. Therefore, the only
holomorphic subbundles of $E_\rho$ that are preserved by the associated connection
$\nabla^\rho$ are $0$ and $E_\rho$ itself. But 
$\ell\, \geq\, 2$ and $W_1\,\not=\, 0$ in \eqref{g1}. So $W_1$
neither $0$ nor $E_\rho$.

In view of the above contradiction, we conclude that
the holomorphic vector bundle $E_\rho$ is semistable.

We will now prove that $E_\rho$ is polystable.

Consider all nonzero coherent analytic subsheaves $V$ of $E_\rho$ such that
\begin{itemize}
\item $V$ is polystable, and

\item $\text{degree}(V)\,=\, 0$.
\end{itemize}
Let
\begin{equation}\label{f1}
{\mathcal F}\, \subset\, E_\rho
\end{equation}
be the coherent analytic subsheaf generated by all $V$ satisfying the above
two conditions. It is know that $\mathcal F$ is polystable with
$\mu({\mathcal F})\,=\, \mu(E_\rho)\,=\, 0$ (see \cite[page 23, Lemma 1.5.5]{HL}).
Therefore, the subsheaf ${\mathcal F}$ is uniquely characterized as follows:
the subsheaf ${\mathcal F}$ is the unique maximal coherent analytic subsheaf of
$E_\rho$ such that
\begin{itemize}
\item $\mathcal F$ is polystable, and

\item $\text{degree}({\mathcal F})\,=\, 0$.
\end{itemize}
Note that the quotient $E_\rho/{\mathcal F}$ is torsionfree, because
if $T\, \subset\, E_\rho/{\mathcal F}$ is the torsion part, then
$\varphi^{-1}(T)\, \subset\, E_\rho$, where
$$
\varphi\, :\, E_\rho\,\longrightarrow\, E_\rho/{\mathcal F}
$$
is the quotient map, also satisfies the above two conditions, while
${\mathcal F}\, \subsetneq \, \varphi^{-1}(T)$ if $T\,\not=\, 0$.

Consider the action $\tau$ of $\text{SL}(2,{\mathbb C})$ on
$E_\rho$ constructed in \eqref{ta}. From the above characterization of
the subsheaf ${\mathcal F}$ in \eqref{f1} it follows immediately that
\begin{equation}\label{f2}
\tau(\text{SL}(2,{\mathbb C})\times {\mathcal F})\,=\,
{\mathcal F}\, .
\end{equation}
As it was done for $W_1$, from \eqref{f2} we conclude that ${\mathcal F}$ is a
holomorphic subbundle of $E_\rho$.

As it was done for $W_1$, from \eqref{f2} it follows that the flat connection
$\nabla^\rho$ on $E_\rho$ preserves the subbundle $\mathcal F$ in \eqref{f1}.
Since $\rho$ is irreducible, either ${\mathcal F}\,=\,0$ or
${\mathcal F}\,=\,E_\rho$. The rank of $\mathcal F$ is
at least one because the semistable vector bundle $E_\rho$ of degree zero
has a nonzero stable subsheaf of degree zero. Therefore, we conclude
that ${\mathcal F}\,=\, E_\rho$. Consequently, $E_\rho$ is polystable.
\end{proof}

We may now ask whether the polystable vector bundle $E_\rho$ in Theorem \ref{thm1}
is stable. The following lemma shows that $E_\rho$ is not stable in general.

Let
\begin{equation}\label{e7}
\delta\, :\, \Gamma\, \hookrightarrow\, \text{SL}(2, {\mathbb C})
\end{equation}
be the inclusion map. This homomorphism $\delta$ is clearly
irreducible. Let $(E_\delta\, ,\nabla^\delta)$ be the
corresponding flat holomorphic vector bundle on $M$.

\begin{lemma}\label{lem1} 
The above holomorphic vector bundle $E_\delta$ is holomorphically trivial.
\end{lemma}

\begin{proof}
Recall that the vector bundle $E_\delta$ is a quotient of $\text{SL}(2,
{\mathbb C})\times{\mathbb C}^2$. Consider the holomorphic map
$$
\text{SL}(2,{\mathbb C})\times {\mathbb C}^2\,\longrightarrow\,
\text{SL}(2,{\mathbb C})\times{\mathbb C}^2
$$
defined by $(g\, ,v)\,\longmapsto\, (g\, ,g(v))$. This map descends to a
holomorphic isomorphism of vector bundles
$$
E_\delta\,\longrightarrow\, M\times {\mathbb C}^2
$$
over $M$. Therefore, this descended homomorphism provides
a holomorphic trivialization of $E_\delta$.
\end{proof}



\begin{thebibliography}{123}

\bibitem[BM]{BM} I. Biswas and A. Mukherjee, Solutions of Strominger system from
unitary representations of cocompact lattices of $\text{SL}(2,{\mathbb C})$,
{\it Comm. Math. Phy.} (in press), http://arxiv.org/abs/1301.0375.

\bibitem[Br]{Br} L. Bruasse, Harder-Narasimhan filtration on non K\"ahler
manifolds, \textit{Internat. Jour. Math.} \textbf{12} (2001), 579--594.

\bibitem[Bu]{Bu} N. P. Buchdahl, Hermitian-Einstein connections and stable
vector bundles over compact complex surfaces, {\it Math. Ann.}
\textbf{280} (1988), 625--648.

\bibitem[HL]{HL} D. Huybrechts and M. Lehn, \textit{The geometry
of moduli spaces of sheaves}, Aspects of Mathematics, E31,
Friedr. Vieweg \& Sohn, Braunschweig, 1997.

\bibitem[Ko]{Ko} S. Kobayashi, \textit{Differential geometry
of complex vector bundles}, Princeton University Press, Princeton,
NJ, Iwanami Shoten, Tokyo, 1987.

\bibitem[La]{La} M. Lackenby, Some 3-manifolds and 3-orbifolds with
large fundamental group, \textit{Proc. Amer. Math. Soc.} \textbf{135}
(2007), 3393--3402.

\bibitem[LY]{LY} J. Li and S.-T. Yau, Hermitian--Yang--Mills connection on
non--K\"ahler manifolds, \textit{Mathematical aspects of string
theory} (San Diego, Calif., 1986), 560--573, Adv. Ser. Math. Phys., 1, World Sci.
Publishing, Singapore, 1987.

\bibitem[Ra]{Ra} C. S. Rajan, Deformations of complex structures on $\Gamma
\backslash {\rm SL}(2,{\mathbf C})$, {\it Proc. Indian Acad. Sci. Math.
Sci.} \textbf{104} (1994), 389--395.

\end{thebibliography}
\end{document}